\documentclass[11pt]{amsart}
\pdfoutput=1
\usepackage{amsmath}
\usepackage{amssymb}
\usepackage{amsthm}
\usepackage[foot]{amsaddr}
\usepackage[margin=1in]{geometry}
\usepackage[hidelinks]{hyperref}
\usepackage{bm}
\usepackage[Symbol]{upgreek}
\usepackage{mathtools}

\usepackage{lmodern}
\usepackage[T1]{fontenc}
\usepackage{microtype}

\DeclareMathOperator{\ddeg}{ddeg}
\DeclareMathOperator{\ndeg}{ndeg}

\def\d {\operatorname{d}}

\def\ac {\operatorname{ac}}
\def\h {\operatorname{h}}

\newcommand{\ca}{\mathcal}

\newcommand{\pconv}{\rightsquigarrow}
\newcommand{\prece}{\preccurlyeq}
\newcommand{\succe}{\succcurlyeq}

\newcommand{\f}{\mathfrak}

\newcommand{\x}{\times}

\newcommand{\ges}{\geqslant}
\newcommand{\les}{\leqslant}

\newcommand{\N}{\mathbb{N}}

\newcommand{\Q}{\mathbb{Q}}
\newcommand{\R}{\mathbb{R}}
\newcommand{\T}{\mathbb{T}}

\DeclareFontFamily{U}{fsy}{}
\DeclareFontShape{U}{fsy}{m}{n}{<->s*[.9]psyr}{}
\DeclareSymbolFont{der@m}{U}{fsy}{m}{n}
\DeclareMathSymbol{\der}{\mathord}{der@m}{182}
\DeclareFontFamily{OMS}{smallo}{}
\DeclareFontShape{OMS}{smallo}{m}{n}{<->s*[.65]cmsy10}{}
\DeclareSymbolFont{smallo@m}{OMS}{smallo}{m}{n}
\DeclareMathSymbol{\cao}{\mathord}{smallo@m}{79}

\newtheorem{thmint}{Theorem}

\newtheorem{lem}{Lemma}[section]
\newtheorem{prop}[lem]{Proposition}
\newtheorem{cor}[lem]{Corollary}
\newtheorem{thm}[lem]{Theorem}

\theoremstyle{definition}

\newtheorem*{defn}{Definition}

\usepackage{etoolbox}

\makeatletter
\patchcmd{\@startsection}
  {\@afterindenttrue}
  {\@afterindentfalse}
  {}{}
\makeatother

\title[Newtonian valued differential fields with arbitrary value group]{Newtonian valued differential fields\\ with arbitrary value group}
\author{Nigel Pynn-Coates}
\address{Department of Mathematics, University of Illinois at Urbana--Champaign, Urbana, IL 61801, USA}
\email{\href{mailto:pynncoa2@illinois.edu}{pynncoa2@illinois.edu}}

\begin{document}

\begin{abstract}
The notion of newtonianity is central to the study of the ordered differential field of logarithmic-exponential transseries done by Aschenbrenner, van den Dries, and van der Hoeven; see \cite[Chapter~14]{adamtt}.
We remove the assumption of divisible value group from two of their results concerning newtonianity, namely the newtonization construction and the equivalence of newtonianity with asymptotic differential-algebraic maximality.
We deduce the uniqueness of immediate differentially algebraic extensions that are asymptotically differential-algebraically maximal.
\end{abstract}

\maketitle

\section{Introduction}

The ordered (valued) differential field of logarithmic-exponential transseries, $\T$, was introduced by Ecalle in his proof of Dulac's conjecture \cite{ecalle1,ecalle2}.
An important notion in studying its model theory is that of ``newtonianity.''
Developed by Aschenbrenner, van den Dries, and van der Hoeven in their breakthrough work on $\T$ \cite{adamtt}, newtonianity is a generalization of henselianity to a certain class of asymptotic valued differential fields.\footnote{There is another generalization of henselianity to the valued differential field setting called ``differential-henselianity,'' introduced by Scanlon in \cite{scanlon} and developed in a more general setting in \cite{adamtt}. Although many asymptotic fields, including $\T$, cannot be differential-henselian, there is a relationship between these two notions; see \cite[\S14.1]{adamtt}.}
Recall that henselianity is a fundamental notion in valuation theory about lifting simple roots from the residue field to the main field.
For example, Hahn fields and $p$-adic fields are henselian.
That the elementary theory of a henselian valued field of residue characteristic 0 is determined by the elementary theories of its residue field and value group is a famous result of Ax--Kochen \cite{ak3} and Ershov \cite{ershov}.

Useful in proving this theorem are the facts that valued fields have henselizations, and, for valued fields of equicharacteristic 0, henselianity is equivalent to algebraic maximality.
Aschenbrenner, van den Dries, and van der Hoeven prove analogous results for a certain class of asymptotic valued differential fields,  namely that such fields have newtonizations, and that newtonianity is equivalent to asymptotic differential-algebraic maximality.
These results played an important role in the proof of quantifier elimination for $\T$.
That newtonianity is the right generalization to this class of asymptotic valued differential fields is evidenced by its centrality to the study of $\T$, including as part of its axiomatization \cite{adamtt}.

In this note, we remove the assumption of divisible value group from these two results, which we do by essentially removing it from the main technical lemma used in their proofs, \cite[Proposition~14.5.1]{adamtt} (Proposition~\ref{adh14.5.1} below).
The idea is to pass to the algebraic closure, where the value group is divisible, but the chief difficulty is in ensuring that minimal differential polynomials of pseudocauchy sequences remain such; this is done in Lemma~\ref{mdpac}.
For technical reasons, henselianity is assumed there, so arguments using henselizations are required in deducing the main results from Proposition~\ref{mainlemmah}.

Before stating the main theorems, we recall some necessary definitions, following the notation and conventions of \cite{adamtt}.
A \emph{valued differential field} is a field $K$ together with a surjective map $v \colon K^{\x} \to \Gamma$ and a map $\der \colon K \to K$ satisfying, for all $x$, $y$ in their domain:
\pagebreak
\begin{enumerate}
	\item[(V1)] $v(xy) = v(x)+v(y)$;
	\item[(V2)] $v(x+y) \ges \min\{v(x), v(y)\}$ whenever $x+y \neq 0$;
	\item[(D1)] $\der(x+y)=\der(x)+\der(y)$;
	\item[(D2)] $\der(xy)=x\der(y)+\der(x)y$.
\end{enumerate}
Here, $\Gamma$ is a (totally) ordered abelian group called the $\emph{value group}$ of $K$.
We introduce a symbol $\infty \notin \Gamma$ and extend the ordering to $\Gamma \cup \{\infty\}$ by $\infty>\Gamma$.
We also set $\infty+\gamma=\gamma+\infty = \infty+\infty \coloneqq \infty$ for all $\gamma \in \Gamma$.
This allows us to extend $v$ to $K$ by setting $v(0) \coloneqq \infty$.
Other notation includes the \emph{valuation ring} $\ca O \coloneqq \{x \in K : v(x) \ges 0\}$ with (unique) maximal ideal $\cao \coloneqq \{x \in K : v(x)>0\}$.
Then $\bm k \coloneqq \ca O/\cao$ is the \emph{residue field} of $K$; throughout, we assume that $\bm k$ has characteristic $0$.
As it is often more intuitive, we define for $a$, $b \in K$:
\[\begin{array}{lc}
a \prece b\ \Leftrightarrow\ v(a)\ges v(b),\qquad a \prec b\ \Leftrightarrow\ v(a)> v(b),\\
a \asymp b\ \Leftrightarrow\ v(a)=v(b),\qquad  
  a\sim b\ \Leftrightarrow\ a-b \prec b.
\end{array}\]
For $a \in K$, we often denote $\der(a)$ by $a'$, and we let $C \coloneqq \{ a \in K : \der(a)=0\}$ be the \emph{constant field} of~$K$.
For the rest of this note, let $K$ be a nontrivially valued differential field with the above notation and conventions.
For other valued differential fields, we decorate these symbols with subscripts.
For example, we say a valued differential field extension $L$ of $K$ is \emph{immediate} if $\Gamma_L = \Gamma$ and $\bm k_L = \bm k$, where we identify $\Gamma$ with a subgroup of $\Gamma_L$ and $\bm k$ with a subfield of $\bm k_L$ in the usual way.
Later, we abbreviate ``valued differential field extension'' by \emph{extension}.

Of course, we are interested in the case when there is some interaction between the valuation and the derivation.
The main basic condition imposed here is that $K$ is \emph{asymptotic}, that is, for all nonzero $f$, $g \in \cao$,
\[ f \prec g\ \iff\ f' \prec g'. \]
For the rest of the introduction, $K$ will be an asymptotic field; we impose further restrictions at the end of this section.
It is immediate that $C \subseteq \ca O$, that is, $K$ has \emph{few constants} in the sense of~\cite[\S4.4]{adamtt}.
Another consequence of being asymptotic is that there is $a \in K^{\x}$ such that $\der\cao \subseteq a\cao$, which means that $\der$ is continuous with respect to the valuation topology on $K$ (see \cite[Lemma~4.4.7 and Corollary~9.1.5]{adamtt}).
If $a=1$ works, then we say $K$ has \emph{small derivation}.

In asymptotic fields, logarithmic derivatives play an important role, so we let $g^\dagger \coloneqq g'/g$ for $g \in K^\x$.
In fact, $v(g^\dagger)$ and $v(g')$ depend only on $vg$, so for $\gamma = vg \in \Gamma^{\neq} \coloneqq \Gamma \setminus \{0\}$, we set $\gamma^\dagger \coloneqq v(g^\dagger)$ and $\gamma' \coloneqq v(g') = \gamma^\dagger+\gamma$.
We then define a valuation $\psi$ on $\Gamma$ in the sense of \cite[\S2.2]{adamtt} by:
\begin{align*}
\psi \colon \Gamma^{\neq} &\to \Gamma\\
\gamma &\mapsto \gamma^\dagger.
\end{align*}
We call $(\Gamma, \psi)$ the \emph{asymptotic couple} of~$K$ and let $\Psi \coloneqq \psi(\Gamma^{\neq})$.
Then $K$ is said to be \emph{ungrounded} if~$\Psi$ has no greatest element.
We say that $K$ is \emph{$H$-asymptotic} if~$\psi$ is a convex valuation as defined in \cite[\S2.4]{adamtt}, that is, $\psi(\alpha) \ges \psi(\beta)$ whenever $0<\alpha \les \beta$.

The broad class of asymptotic fields includes differential-valued fields, introduced by Rosenlicht~\cite{rosendval}, such as $\T$ and Hardy fields extending~$\R$.
We say~$K$ is \emph{differential-valued} (\emph{$\d$-valued} for short) if $\ca O = C + \cao$.
We are concerned here mainly with ungrounded $H$-asymptotic fields that are also $\d$-valued, but we occasionally relax this last condition a little.
Ungrounded asymptotic fields are in fact \emph{pre-differential-valued} (\emph{pre-$\d$-valued} for short): for all $f$, $g \in K^{\x}$, $f' \prec g^\dagger$ whenever $f \prece 1$ and $g \prec 1$ \cite[Corollary 10.1.3]{adamtt}.

Apart from being an ungrounded $H$-asymptotic $\d$-valued field, two of the main properties of $\T$ are newtonianity, already mentioned, and $\upomega$-freeness.
The former has a somewhat technical definition that we delay until \S\ref{prelim}.
But granted that, we can define newtonizations; in this definition, we require~$K$ to be ungrounded and $H$-asymptotic.
We say an ungrounded $H$-asymptotic extension~$L$ of~$K$ is a \emph{newtonization} of $K$ if it is newtonian and embeds over $K$ into every ungrounded $H$-asymptotic extension of $K$ that is newtonian.

\begin{defn}
We say an ungrounded $H$-asymptotic field $K$ is \emph{$\upomega$-free} if for all $f \in K$, there is $g \in K$ such that $g \succ 1$ and
\[
f+\big(2(-g^{\dagger\dagger})'+(g^{\dagger\dagger})^2\big) \succe (g^{\dagger})^2.
\]
\end{defn}
Thus $\upomega$-freeness can be expressed as a universal-existential sentence in the language of valued differential fields, but it is also equivalent to the absence of a pseudolimit of a certain pseudocauchy sequence $(\upomega_\rho)$ related to iterated logarithms, hence the ``free'' in the name.
In $\T$, the sequence $(\upomega_n)$ is indexed by $\N$ and defined by:
\[\upomega_n\ =\ \frac{1}{\left(\ell_0\right)^2} + \frac{1}{\left(\ell_0\ell_1\right)^2} + \dots + \frac{1}{\left(\ell_0\ell_1\dots\ell_n\right)^2},\]
where $\ell_0 \coloneqq x$ and $\ell_{n+1} \coloneqq \log\ell_n$; here and throughout we let $n$ and $r$ range over $\N = \{0, 1, 2, \dots\}$.
See \cite[Corollary~11.7.8]{adamtt} and the surrounding pages for more on these equivalent definitions.
The definition of $\upomega$-freeness is not specifically used in this note.

We can now state the main theorems.
The first connects newtonianity to asymptotic $\d$-algebraic maximality:
$K$ is said to be \emph{asymptotically differential-algebraically maximal} (\emph{asymptotically $\d$-algebraically maximal} for short) if it has no proper immediate $\d$-algebraic extension that is asymptotic.
Here and throughout, we abbreviate ``differentially algebraic'' by \emph{$\d$-algebraic}.
Note that any $K$ satisfying the assumptions of the first theorem is in fact $\d$-valued by \cite[Lemma~14.2.5]{adamtt} and Lemma~\ref{omegalambda}.
\begin{thmint}\label{mainequiv}
If $K$ is an ungrounded $H$-asymptotic field that is $\upomega$-free and newtonian, then it is asymptotically $\d$-algebraically maximal.
\end{thmint}

\begin{thmint}\label{mainembed}
If $K$ is an ungrounded $H$-asymptotic field that is $\upomega$-free and $\d$-valued, then it has a newtonization.
\end{thmint}
These are Theorems~\ref{newttoasympdalgmax} and \ref{newtonization} below.
The case that the value group is divisible is covered by \cite[Theorem~14.0.2 and Corollary~14.5.4]{adamtt}.
The converse of Theorem~\ref{mainequiv} holds for ungrounded $H$-asymptotic $K$ that are $\uplambda$-free, a weaker notion than $\upomega$-freeness to be defined in \S\ref{prelim}; see  Theorem~\ref{adh14.0.1} (\cite[Theorem~14.0.1]{adamtt}).

Finally, the next theorem is a corollary of the previous two results, but we provide an alternative proof in \S\ref{unique}.
The case that the value group is divisible is not stated in \cite{adamtt}, but follows from the corresponding results.
\begin{thmint}\label{mainunique}
If $K$ is an ungrounded $H$-asymptotic field that is $\upomega$-free and $\d$-valued, then any two immediate $\d$-algebraic extensions of $K$ that are asymptotically $\d$-algebraically maximal are isomorphic over $K$.
\end{thmint}

\subsection*{Assumptions}
As we are primarily concerned with ungrounded $H$-asymptotic fields, to avoid repetition we assume throughout that $K$ is an ungrounded $H$-asymptotic field with asymptotic couple $(\Gamma, \psi)$ and $\Psi \coloneqq \psi(\Gamma^{\neq})$; recall we also assume that $\Gamma\neq\{0\}$.
Moreover, we assume that any (valued differential field) extension of~$K$ is also $H$-asymptotic and ungrounded.

\section{Preliminaries}\label{prelim}

For the reader's convenience, we review here some definitions and results from \cite{adamtt} that are used later.
Most of these are from Chapters 9, 11, 13, and 14, and although proofs are omitted, cross references are provided.

\subsection*{More on asymptotic fields}
The first lemma does not require the $H$-asymptotic or ungrounded assumptions on $K$ or $L$ and is valid for any asymptotic $K$ and $L$.
\begin{lem}[{\cite[9.1.2]{adamtt}}]\label{adh9.1.2}
If $K$ is $\d$-valued and $L$ is an asymptotic extension of $K$ with $\bm k_L = \bm k$, then $L$ is $\d$-valued and $C_L=C$.
\end{lem}

In the main results we will assume that $K$ is $\upomega$-free, but at various places we weaken that assumption for generality.
For instance, $K$ has \emph{asymptotic integration} if $\Gamma=(\Gamma^{\neq})'$.
We also consider this property for the algebraic closure $K^{\ac}$ of $K$.
First, we equip $K^{\ac}$ with any valuation extending that of $K$ and the unique derivation extending that of $K$; this determines $K^{\ac}$ as a valued differential field extension of $K$ up to isomorphism over $K$.
Its value group is the divisible hull $\Q\Gamma$ of $\Gamma$ and its residue field is the algebraic closure $\bm k^{\ac}$ of $\bm k$.
In fact, $K^{\ac}$ is an asymptotic field \cite[Proposition~9.5.3]{adamtt} and we extend $\psi$ to $\Q\Gamma^{\neq}$ by $\psi(q\gamma) = \psi(\gamma)$ for all $q \in \Q^{\x}$ and all $\gamma \in \Gamma^{\neq}$.
Thus $\psi(\Q\Gamma^{\neq})=\psi(\Gamma^{\neq})$, so $K^{\ac}$ is $H$-asymptotic and ungrounded.
If $K$ is $\d$-valued, then so is $K^{\ac}$ \cite[Corollary~10.1.23]{adamtt}. 
We say $K$ has \emph{rational asymptotic integration} if $K^{\ac}$ has asymptotic integration, that is, $\Q\Gamma = (\Q\Gamma^{\neq})'$.

\begin{defn}
We say $K$ is \emph{$\uplambda$-free} if for all $f \in K$, there is $g \in K$ such that $g \succ 1$ and $f-g^{\dagger\dagger} \succe g^{\dagger}$.
\end{defn}
Thus $\uplambda$-freeness can be expressed as a universal-existential sentence in the language of valued differential fields, but it is also equivalent to the absence of a pseudolimit of a certain pseudocauchy sequence $(\uplambda_\rho)$ related to iterated logarithms, hence the ``free'' in the name.
In $\T$, the sequence $(\uplambda_n)$ is indexed by $\N$ and defined by:
\[\uplambda_n\ =\ -\left(\ell_n^{\dagger\dagger}\right)\ =\ \frac{1}{\ell_0} + \frac{1}{\ell_0\ell_1} + \dots + \frac{1}{\ell_0\ell_1\dots\ell_n}.\]
See \cite[Corollary~11.6.1]{adamtt} and the surrounding pages for more on these equivalent definitions.
This definition is not specifically used.
Here is the relationship among these notions,
collecting \cite[Corollary~11.6.8]{adamtt} and \cite[Corollary~11.7.3]{adamtt}.

\begin{lem}\label{omegalambda}
We have
\[ \upomega\text{-free} \implies \uplambda\text{-free} \implies \text{rational asymptotic integration} \implies \text{asymptotic integration}.\]
\end{lem}

\begin{thm}[{\cite[14.0.1]{adamtt}}]\label{adh14.0.1}
If $K$ is $\uplambda$-free and asymptotically $\d$-algebraically maximal, then it is $\upomega$-free and newtonian.
\end{thm}

Both $\uplambda$-freeness and $\upomega$-freeness are preserved under algebraic extensions, and even more is true for $\upomega$-freeness.
In particular, these properties are preserved when passing to the henselization, a fact that is important in the main results.

\begin{lem}[{\cite[11.6.8]{adamtt}}]\label{adh11.6.8}
The algebraic closure $K^{\ac}$ of $K$ is $\uplambda$-free if and only if $K$ is $\uplambda$-free.
\end{lem}

\begin{lem}[{\cite[11.7.23]{adamtt}}]\label{adh11.7.23}
The algebraic closure $K^{\ac}$ of $K$ is $\upomega$-free if and only if $K$ is $\upomega$-free.
\end{lem}

\begin{thm}[{\cite[13.6.1]{adamtt}}]\label{adh13.6.1}
If $K$ is $\upomega$-free and $L$ is a pre-$\d$-valued extension of $K$ that is $\d$-algebraic over $K$, then $L$ is $\upomega$-free.
\end{thm}

\subsection*{Differential polynomials and dominant degree}
In this subsection, let $E$ be a differential field of characteristic 0.
We let $E\{Y\} \coloneqq E[Y, Y', Y'', \dots]$ be the differential polynomial ring over $E$ (with derivation extending that of $E$) and set $E\{Y\}^{\neq} \coloneqq E\{Y\}\setminus\{0\}$.
Let $P \in E\{Y\}^{\neq}$.
Then the \emph{order} of $P$ is the least $r$ such that $P \in E[Y, Y', \dots, Y^{(r)}]$.
Let $r_P$ be the order of $P$. Then the \emph{degree} of $P$ is its total degree as an element of $E[Y, Y', \dots, Y^{(r_P)}]$, denoted by $\deg P$.
Let $s_P$ be the degree of $P$ in $Y^{(r_P)}$ and $t_P=\deg P$.
Then the \emph{complexity} of $P$ is the ordered triple $(r_P, s_P, t_P)$ and is denoted by $c(P)$.
The natural decomposition of $P$ is
$P = \sum_{\bm i} P_{\bm i} Y^{\bm i}$, where $\bm i \in \N^{1+r_P}$ and $Y^{\bm i}=Y^{i_0}(Y')^{i_1}\dots(Y^{(r_P)})^{i_{r_P}}$.

Now suppose that $E$ is also equipped with a valuation $v_E$.
Then we extend $v_E$ to $E\{Y\}^{\neq}$ by $v_E(P) \coloneqq \min_{\bm i}\{v_E(P_{\bm i})\}$ and use $\prece$, $\prec$, $\asymp$, and $\sim$ as for $E$.
We often use the \emph{multiplicative conjugate} $P_{\x a} \coloneqq P(aY)$ and the \emph{additive conjugate} $P_{+a} \coloneqq P(a+Y)$, for $a \in E$.
For more on such conjugation, see \cite[\S4.3]{adamtt}.

Suppose additionally that $E$ has small derivation.
Then we associate $\f d_P \in E^{\x}$ to $P$ such that $\f d_P \asymp P$ and for any $Q \in E\{Y\}^{\neq}$, $\f d_P = \f d_Q$ whenever $P \sim Q$.
The \emph{dominant part} of $P$ is $D_P \coloneqq \sum_{\bm i} \overline{\f d_P^{-1}P_{\bm i}} Y^{\bm i} \in \bm k_E\{Y\}^{\neq}$, and  $\ddeg P \coloneqq \deg D_P$ is called the \emph{dominant degree} of $P$.
Here, we need only the definition of dominant degree in order to define newton degree in the next subsection, but it and related concepts play an important in the theory of valued differential fields with small derivation.
For more on these notions, see \cite[\S6.6]{adamtt}.


\subsection*{Compositional conjugation and newton degree}
We recall the notion of newton degree from \cite[\S11.1 and \S11.2]{adamtt}, a more subtle version of dominant degree for asymptotic fields that may not have small derivation.\footnote{Newton degree has since been extended to valued differential fields with continuous derivation in \cite{maximext}.}
We say $\phi \in K^{\x}$ is \emph{active} (\emph{in $K$}) if $v\phi \in \Psi^\downarrow$, where $\Psi^\downarrow$ denotes the downward closure of $\Psi$ in $\Gamma$.
Below, we let $\phi$ range over active elements of $K^\x$.
To $K$, we associate the valued differential field $K^{\phi}$, which is simply the field $K$ with the derivation $\phi^{-1}\der$ and unchanged valuation; it is still $H$-asymptotic and ungrounded, and moreover it has small derivation by \cite[Lemma~9.2.9]{adamtt}.
We call $K^{\phi}$ the \emph{compositional conjugate of $K$ by $\phi$}.

This leads to the ring $K^{\phi}\{Y\}$ of differential polynomials over $K^{\phi}$, which is viewed as a differential ring with derivation extending $\phi^{-1}\der$.
We then have a ring isomorphism $K\{Y\} \to K^{\phi}\{Y\}$ given by associating to $P \in K\{Y\}$ an appropriate element $P^{\phi} \in K^{\phi}\{Y\}$, called the \emph{compositional conjugate of $P$ by $\phi$}, with the property that $P^{\phi}(y) = P(y)$ for all $y \in K$.
The details of this map are not used here and can be found in \cite[\S5.7]{adamtt}, but it is the identity on the common subring $K[Y] = K^{\phi}[Y]$ of $K\{Y\}$ and $K^{\phi}\{Y\}$.
What is important here is that $\ddeg P^{\phi}$ eventually stabilizes, that is, there is an active $\phi_0 \in K^{\x}$ such that for all $\phi \prece \phi_0$, $\ddeg P^\phi=\ddeg P^{\phi_0}$.
We call this eventual value of $\ddeg P^\phi$ the \emph{newton degree of $P$} and denote it by $\ndeg P$.
With this, we can finally define newtonianity.

\begin{defn}
We call $K$ \emph{newtonian} if each $P \in K\{Y\}$ with $\ndeg P = 1$ has a zero in $\ca O$.
\end{defn}

We know that if $K$ has an immediate newtonian extension, it must have a minimal one.

\begin{lem}[{\cite[14.1.9]{adamtt}}]\label{adh14.1.9}
If $K$ has an immediate newtonian extension, then $K$ has a $\d$-algebraic such extension that has no proper newtonian differential subfield containing $K$.
\end{lem}

Newton degree is connected to pseudocauchy sequences in an important way.
(We abbreviate ``pseudocauchy sequence'' by \emph{pc-sequence};
basic facts about them can be found in \cite[\S2.2 and \S3.2]{adamtt}.)
We associate to each pc-sequence $(a_\rho)$ in $K$ its \emph{cut} (\emph{in $K$}), denoted by $c_K(a_\rho)$, such that if $(b_\lambda)$ is a pc-sequence in $K$, then
\[c_K(a_\rho) = c_K(b_\lambda)\ \iff\ (b_\lambda)\ \text{is equivalent to}\ (a_\rho).\]
In the rest of the paper, let $(a_\rho)$ be a pc-sequence in $K$ with $\bm a = c_K(a_\rho)$.
If $L$ is an extension of $K$, then we let $\bm a_L$ denote $c_L(a_\rho)$.
In order to define newton degree in a cut, we set, for any $\gamma \in \Gamma$ and $P \in K\{Y\}^{\neq}$,
\[
\ndeg_{\ges \gamma} P\ \coloneqq\ \max\{ \ndeg P_{\x g} : g \in K^{\x}, vg \ges \gamma \}.
\]
Note that $\ndeg_{\ges \gamma} P = \ndeg P_{\x g}$ for any $g \in K^{\x}$ with $vg = \gamma$.
With $\gamma_\rho \coloneqq v(a_{\rho+1}-a_{\rho})$ (where $\rho+1$ denotes the successor of $\rho$ in the well-ordered set of indices), there is $\rho_0$ such that
\[
\ndeg_{\ges \gamma_\rho} P_{+a_\rho}\ =\ \ndeg_{\ges \gamma_{\rho_0}} P_{+a_{\rho_0}}
\]
for all $\rho>\rho_0$, and this value depends only on $c_K(a_\rho)$, not the choice of pc-sequence (see~\cite[Lemma~11.2.11]{adamtt}).
In the next definition and two lemmas, $P \in K\{Y\}^{\neq}$.

\begin{defn}
The \emph{newton degree of $P$ in the cut of $(a_\rho)$} is the eventual value of $\ndeg_{\ges \gamma_\rho} P_{+a_\rho}$, denoted by $\ndeg_{\bm a} P$.
\end{defn}

Note that $c_K(a_\rho+y)$ for $y \in K$ depends only on $\bm a$ and $y$, so we let $\bm a+y$ denote $c_K(a_\rho+y)$. Similarly, $c_K(a_\rho y)$ for $y \in K^\times$ depends only on $\bm a$ and $y$, so we let $\bm a\cdot y$ denote $c_K(a_\rho y)$.

\pagebreak
\begin{lem}[{\cite[11.2.12]{adamtt}}]\label{ndegbasic}
Newton degree in a cut has the following properties:
\begin{enumerate}
\item $\ndeg_{\bm a} P \les \deg P$;
\item $\ndeg_{\bm a} P^{\phi} = \ndeg_{\bm a} P$;
\item $\ndeg_{\bm a} P_{+y} = \ndeg_{\bm a+y} P$ for $y \in K$;
\item if $y \in K$ and $vy$ is in the width of $(a_\rho)$, then $\ndeg_{\bm a} P_{+y} = \ndeg_{\bm a} P$;
\item $\ndeg_{\bm a} P_{\times y} = \ndeg_{\bm a\cdot y} P$ for $y \in K^\times$;
\item if $Q \in K\{Y\}^{\neq}$, then $\ndeg_{\bm a} PQ=\ndeg_{\bm a}P + \ndeg_{\bm a} Q$;
\item if there is a pseudolimit $\ell$ of $(a_\rho)$ in an extension of $K$ with $P(\ell)=0$, then $\ndeg_{\bm a} P \ges 1$;
\item if $L$ is an extension of $K$ with $\Psi$ cofinal in $\Psi_L$, then $\ndeg_{\bm a} P=\ndeg_{\bm a_L} P$.
\end{enumerate}
\end{lem}

Recall that $P \in K\{Y\}$ is a \emph{minimal differential polynomial of $(a_\rho)$ over $K$} if $P(b_\lambda) \pconv 0$ for some pc-sequence $(b_\lambda)$ in $K$ equivalent to $(a_\rho)$ and $c(P)$ is minimal among differential polynomials with this property; see \cite[\S4.4]{adamtt}.
Note that then $P \notin K$.
We say $K$ is \emph{strongly newtonian} if it is newtonian and, for every $P \in K\{Y\}$ and every pc-sequence $(a_\rho)$ in $K$ with minimal differential polynomial $P$ over $K$, $\ndeg_{\bm a} P=1$.
Here is the usefulness of this notion.
\begin{lem}[{\cite[14.1.10]{adamtt}}]\label{adh14.1.10}
If $K$ is newtonian and $\ndeg_{\bm a} P=1$, then $P(a)=0$ for some $a \in K$ with $a_\rho \pconv a$.
\end{lem}

\subsection*{Constructing immediate asymptotic extensions}
Suppose for the rest of this section that $K$ has rational asymptotic integration and $P \in K\{Y\}^{\neq}$.
We recall some lemmas from \cite[\S11.4]{adamtt} on how to construct immediate asymptotic extensions of such fields with appropriate embedding properties.
The first lemma is about evaluating differential polynomials at pc-sequences.
Although the statement of \cite[Lemma~11.3.8]{adamtt} is slightly less general than that given here, the same proof gives the following.
Here, $\Gamma^{<} \coloneqq \{\gamma : \gamma<0\}$.

\begin{lem}[{\cite[11.3.8]{adamtt}}]\label{adh11.3.8}
Let $E$ be an extension of $K$ with rational asymptotic integration and such that $\Gamma^{<}$ is cofinal in $\Gamma_{E}^{<}$.
Suppose $(a_\rho)$ has a pseudolimit $\ell$ in $E$ and let $G \in E\{Y\} \setminus E$.
Then there a pc-sequence $(b_\lambda)$ in $K$ equivalent to $(a_\rho)$ such that $\big(G(b_\lambda)\big)$ is a pc-sequence in $E$ with $G(b_\lambda) \pconv G(\ell)$.
\end{lem}

Let $\ell \notin K$ be an element in an extension of $K$ such that $v(\ell - K) \coloneqq \{v(\ell-x) : x \in K\}$ has no largest element (equivalently, $\ell$ is the pseudolimit of some divergent pc-sequence in $K$).
We say that $P$ \emph{vanishes at $(K,\ell)$} if for all $a \in K$ and $\f v \in K^\x$ with $a-\ell \prec \f v$, $\ndeg_{\prec \f v} P_{+a} \ges 1$, where $\ndeg_{\prec \f v} P_{+a} \coloneqq \max\{ \ndeg P_{+a, \x g} : g \in K^{\x}, g \prec \f v\}$.
Then $Z(K, \ell)$ denotes the set of nonzero differential polynomials over $K$ vanishing at $(K,\ell)$.

\begin{lem}[{\cite[11.4.7]{adamtt}}]\label{adh11.4.7}
Suppose $Z(K, \ell)=\emptyset$.
Then $P(\ell) \neq 0$ for all $P$ and $K\langle \ell \rangle$ is an immediate extension of $K$.
If $g$ in an extension $L$ of $K$ satisfies $v(a-g)=v(a-\ell)$ for all $a \in K$, then there is a unique valued differential field embedding $K\langle\ell\rangle \to L$ over $K$ sending $\ell$ to $g$.
\end{lem}

\begin{lem}[{\cite[11.4.8]{adamtt}}]\label{adh11.4.8}
Suppose $Z(K, \ell)\neq\emptyset$ and $P \in Z(K, \ell)$ has minimal complexity.
Then $K$ has an immediate extension $K\langle f\rangle$ with $P(f)=0$ and $v(a-f)=v(a-\ell)$ for all $a \in K$.
Moreover, for any extension $L$ of $K$ with $g \in L$ satisfying $P(g)=0$ and $v(a-g)=v(a-\ell)$ for all $a \in K$, there is a unique valued differential field embedding $K\langle f\rangle \to L$ over $K$ sending $f$ to $g$.
\end{lem}

\begin{lem}[{\cite[11.4.11]{adamtt}}]\label{adh11.4.11}
Suppose $(a_\rho)$ is divergent in $K$ with $a_\rho \pconv \ell$.
If $\big(P(a_\rho)\big)$ is a pc-sequence such that $P(a_\rho) \pconv 0$, then $P \in Z(K, \ell)$.
\end{lem}

The notion of vanishing is connected to newton degree in a cut.
Although the first lemma is not directly used, it is worth pointing out.
\begin{lem}[{\cite[11.4.12]{adamtt}}]
Suppose $(a_\rho)$ is divergent in $K$ with $a_\rho \pconv \ell$.
Then
\[\ndeg_{\bm a} P\ =\ \min\{\ndeg_{\prec \mathfrak{v}} P_{+a} : a-\ell \prec \mathfrak{v}\}.\]
In particular, $\ndeg_{\bm a} P \ges 1 \iff P \in Z(K, \ell)$.
\end{lem}

\begin{cor}[{\cite[11.4.13]{adamtt}}]\label{adh11.4.13}
Suppose $(a_\rho)$ is divergent in $K$.
Then the following are equivalent:
\begin{enumerate}
	\item $P \in Z(K, \ell)$ has minimal complexity in $Z(K, \ell)$;
	\item $P$ is a minimal differential polynomial of $(a_\rho)$ over $K$.
\end{enumerate}
\end{cor}

\section{Removing divisibility}
The following proposition is the main tool in proving the desired results in the case $\Gamma$ is divisible.
Its proof uses the differential newton diagram method of \cite{adamtt}, as well as the rather technical notion of unravellers, and is spread over several chapters of \cite{adamtt}.
\begin{prop}[{\cite[14.5.1]{adamtt}}]\label{adh14.5.1}
Suppose $K$ is $\upomega$-free and $\d$-valued with divisible $\Gamma$.
Let $P$ be a minimal differential polynomial of $(a_\rho)$ over $K$.
Then $\ndeg_{\bm a} P=1$.
\end{prop}

The following is the main lemma that will allow us to replace the divisibility of $\Gamma$ in the previous result with the assumption that $K$ is henselian.
The main theorems are then proven using the new proposition and arguments with henselizations.

\begin{lem}\label{mdpac}
Suppose $K$ is henselian and has rational asymptotic integration.
Let $P$ be a minimal differential polynomial of $(a_\rho)$ over $K$.
Then $P$ remains a minimal differential polynomial of $(a_\rho)$ over the algebraic closure $K^{\ac}$ of $K$.
\end{lem}
\begin{proof}
Recall that $K^{\ac}$ is an extension of $K$ with $\Psi_{K^{\ac}}=\Psi$.
Since $K$ has rational asymptotic integration, so does $K^{\ac}$.
Note also that  $\Gamma^{\neq}$ has no smallest archimedean class because $\psi$ is constant on each archimedean class and $K$ is ungrounded.
It follows that $\Gamma^{<}$ is cofinal in $\Gamma^{<}_{K^{\ac}} = (\Q\Gamma)^{<}$.

We may suppose that $(a_\rho)$ is divergent in $K$, the other case being trivial.
Then $(a_\rho)$ must still be divergent in $K^{\ac}$:
If it had a pseudolimit $a \in K^{\ac}$, then we would have $Q(a_\rho) \pconv 0$, where $Q \in K[Y]$ is the minimum polynomial of $a$ over $K$ (see \cite[Proposition~3.2.1]{adamtt}).
But since $K$ is henselian, it is algebraically maximal (see \cite[Corollary~3.3.21]{adamtt}), and then $(a_\rho)$ would have a pseudolimit in $K$.

Now, suppose to the contrary that $Q$ is a minimal differential polynomial of $(a_\rho)$ over $K^{\ac}$ with $c(Q)<c(P)$. 
Take an extension $L \subseteq K^{\ac}$ of $K$ with $Q \in L\{Y\}$ and $[L : K] = n < \infty$.
Since $K$ is henselian, $[L:K]=[\Gamma_L:\Gamma] \cdot [\bm k_L : \bm k]$ (see \cite[Corollary~3.3.49]{adamtt}), and thus we have a valuation basis $\ca B=\{e_1, \dots, e_n\}$ of $L$ over $K$ (see \cite[Proposition~3.1.7]{adamtt}).
That is, $\ca B$ is a basis of $L$ over $K$, and for all $a_1,\dots,a_n \in K$,
\[v\left(\sum_{i=1}^n a_i e_i\right)\ =\ \min_{1 \les i \les n} v(a_i e_i).\]
Then by expressing the coefficients of $Q$ in terms of the valuation basis,
\[Q(Y)\ =\ \sum_{i=1}^n R_i(Y) \cdot e_i,\]
where $R_i(Y) \in K\{Y\}$ for $1 \les i \les n$.

Since $Q$ is a minimal differential polynomial of $(a_\rho)$ over $K^{\ac}$, by Corollary~\ref{adh11.4.13} and Lemma~\ref{adh11.4.8} we have an immediate extension $K^{\ac}\langle a\rangle$ of $K^{\ac}$ with $a_\rho \pconv a$ and $Q(a)=0$.
Then by Lemma~\ref{adh11.3.8}, there is a pc-sequence $(b_\lambda)$ in $K$ equivalent to $(a_\rho)$ such that $Q(b_\lambda) \pconv Q(a)=0$.
Finally, after passing to a cofinal subsequence, we have $i$ with $Q(b_\lambda) \asymp R_i(b_\lambda) \cdot e_i$ for all $\lambda$.
Then $R_i(b_\lambda) \pconv 0$ and $c(R_i)<c(P)$, contradicting the minimality of $P$.
\end{proof}

Note that in the next results, we assume that $K$ is $\upomega$-free, and so has rational asymptotic integration by Lemma~\ref{adh11.6.8}.
This makes the previous lemma available.

\begin{prop}\label{mainlemmah}
Suppose $K$ is $\upomega$-free, $\d$-valued, and henselian.
Let $P$ be a minimal differential polynomial of $(a_\rho)$ over $K$.
Then $\ndeg_{\bm a} P=1$.
\end{prop}
\begin{proof}
By the previous lemma, $P$ remains a minimal differential polynomial of $(a_\rho)$ over $K^{\ac}$.
Also,~$K^{\ac}$ is $\upomega$-free by Lemma~\ref{adh11.7.23} and $\d$-valued by \cite[Corollary~10.1.23]{adamtt}.
But then $\ndeg_{\bm a_{K^{\ac}}} P=1$ by Proposition~\ref{adh14.5.1}, and since $\Psi_{K^{\ac}}=\Psi$, Lemma~\ref{ndegbasic}(viii) gives $\ndeg_{\bm a} P=1$.
\end{proof}

Since being newtonian implies being henselian, we then can improve \cite[Corollary~14.5.2]{adamtt} using the same proof; this result contains Theorem~\ref{mainequiv}.

\begin{thm}\label{newttoasympdalgmax}
Suppose $K$ is $\upomega$-free.
Then the following are equivalent:
\begin{enumerate}
\item $K$ is newtonian;
\item $K$ is strongly newtonian;
\item $K$ is asymptotically $\d$-algebraically maximal.
\end{enumerate}
\end{thm}

The next result gives Theorem~\ref{mainembed}. Its proof is the same as that of \cite[Corollary~14.5.4]{adamtt} after reducing to the henselian case, but is given for completeness.
\begin{thm}\label{newtonization}
Suppose $K$ is $\upomega$-free and $\d$-valued, and let $L$ be an immediate $\d$-algebraic extension of $K$ that is newtonian.
Then $L$ is a newtonization of $K$, and any other newtonization of $K$ is isomorphic to $L$ over $K$.
\end{thm}
\begin{proof}
Let $E$ be an extension of $K$ that is newtonian.
Note that the henselization $K^{\h}$ of $K$ is an immediate extension of $K$, so is $\d$-valued by Lemma~\ref{adh9.1.2}.
By Lemma~\ref{adh11.7.23}, $K^{\h}$ is $\upomega$-free.
Since $E$ is newtonian, it is henselian, so by embedding $K^{\h}$ in $E$, we may assume that $K$ is henselian.

The case $K=L$ being trivial, we may suppose that $K \neq L$.
It is sufficient to find $a \in L \setminus K$ such that $K\langle a\rangle$ embeds into $E$ over $K$, since any extension $F \subseteq L$ of $K$ is still $\upomega$-free and $\d$-valued by Theorem~\ref{adh13.6.1} and Lemma~\ref{adh9.1.2}, respectively.
Let $\ell \in L \setminus K$, so there is a divergent pc-sequence $(a_\rho)$ in~$K$ with $a_\rho \pconv \ell$.
Since $L$ is a $\d$-algebraic extension of $K$, $Z(K, \ell) \neq \emptyset$ by Lemma~\ref{adh11.4.7}, and so there is a minimal differential polynomial $P$ of $(a_\rho)$ over $K$ by Corollary~\ref{adh11.4.13}.
Hence by Proposition~\ref{mainlemmah}, $\ndeg_{\bm a} P=1$.
By the $\upomega$-freeness of $K$ and \cite[Corollary~13.6.13]{adamtt}, newton degree remains the same in $L$ and $E$, so $\ndeg_{\bm a_L} P = \ndeg_{\bm a_E} P = 1$.
Then Lemma~\ref{adh14.1.10} gives $a \in L \setminus K$ with $a_\rho \pconv a$ and $P(a)=0$, and $b \in E \setminus K$ with $a_\rho \pconv b$ and $P(b)=0$.
Then Lemma~\ref{adh11.4.8} gives an embedding of $K\langle a\rangle$ into $E$ over $K$.

The uniqueness of $L$ follows from its embedding property and Lemma~\ref{adh14.1.9}.
\end{proof}

Here is one last corollary, an improvement of \cite[Corollary~14.5.6]{adamtt} with the same proof.
\begin{cor}
Suppose $K$ is $\upomega$-free and $\d$-valued.
If $L = K(C_L)$ is an algebraic extension of $K$ that is newtonian, then so is $K$.
\end{cor}

\section{Uniqueness}\label{unique}

For $\upomega$-free $\d$-valued $K$, it follows from Theorems~\ref{newttoasympdalgmax} and \ref{newtonization} that any two immediate $\d$-algebraic extensions of $K$ that are asymptotically $\d$-algebraically maximal are isomorphic over $K$.
In this section, we provide an alternative argument, making explicit some ideas only tacitly present in \cite{adamtt}.

\begin{lem}
Suppose $K$ has rational asymptotic integration.
Then every pc-sequence in $K$ with a minimal differential polynomial over $K$ has a pseudolimit in $K$ if and only if $K$ is asymptotically $\d$-algebraically maximal.
\end{lem}
\begin{proof}
Suppose first that there is $\ell \notin K$ in some immediate $\d$-algebraic extension of $K$ and let $P$ be its minimal annihilator over $K$.
There is a divergent pc-sequence $(a_\rho)$ in $K$ with $a_\rho \pconv \ell$, so by Lemma~\ref{adh11.3.8}, there is an equivalent pc-sequence $(b_\lambda)$ in $K$ with $P(b_\lambda) \pconv P(\ell)=0$.
Taking such a $P$ with minimal complexity, we obtain a minimal differential polynomial of $(a_\rho)$ over $K$.

For the other direction, let $(a_\rho)$ be a divergent pc-sequence in $K$ with minimal differential polynomial $P$ over $K$.
By \cite[Lemma~2.2.5]{adamtt}, or rather its proof, $(a_\rho)$ has a pseudolimit $\ell$ in an extension of $K$.
By Corollary~\ref{adh11.4.13}, $Z(K, \ell) \neq \emptyset$, so Lemma~\ref{adh11.4.8} gives a proper immediate extension of $K$.
\end{proof}

\begin{lem}\label{asympdalgmaxplimroot}
Suppose $K$ is $\upomega$-free, $\d$-valued, and henselian.
Let $P$ be a minimal differential polynomial of $(a_\rho)$ over $K$.
Let $L$ be a $\d$-valued extension of $K$ that is $\uplambda$-free and asymptotically $\d$-algebraically maximal.
Then $a_\rho \pconv b$ and $P(b)=0$ for some $b \in L$.
\end{lem}
\begin{proof}
By Proposition~\ref{mainlemmah}, $\ndeg_{\bm a} P=1$.
By the $\upomega$-freeness of $K$ and \cite[Corollary~13.6.13]{adamtt}, $\ndeg_{\bm a_L} P = 1$.
By Theorem~\ref{adh14.0.1}, $L$ is newtonian, so we get $a_\rho \pconv b$ and $P(b)=0$ for some $b \in L$ from Lemma~\ref{adh14.1.10}.
\end{proof}

\begin{thm}\label{maximmed}
Suppose $K$ is $\upomega$-free and $\d$-valued.
Then any two immediate $\d$-algebraic extensions of $K$ that are asymptotically $\d$-algebraically maximal are isomorphic over $K$.
\end{thm}
\begin{proof}
Let $L_0$ and $L_1$ be immediate $\d$-algebraic extensions of $K$ that are asymptotically $\d$-algebraically maximal.
Note that they are both $\upomega$-free and $\d$-valued by Theorem~\ref{adh13.6.1} and Lemma~\ref{adh9.1.2}, respectively.
By Zorn's lemma we have a maximal isomorphism $\mu \colon F_0 \cong_K F_1$ between asymptotic valued differential subfields $F_i\supseteq K$ of $L_i$ for $i=0,1$, where ``maximal'' means that $\mu$ does not extend to an isomorphism between strictly larger such subfields.
As before, $F_i$ is $\upomega$-free and $\d$-valued for $i=0,1$.
Next, they must be henselian, because the henselization of $F_i$ in $L_i$ is an algebraic field extension of $F_i$, and thus a valued differential subfield of $L_i$ that is $\upomega$-free and $\d$-valued for $i=0,1$.

Now suppose towards a contradiction that $F_0 \neq L_0$ (equivalently, $F_1 \neq L_1$).
Then $F_0$ is not asymptotically $\d$-algebraically maximal, so we have a divergent pc-sequence $(a_\rho)$ in $F_0$ with a minimal differential polynomial $P$ over $F_0$.
Then Lemma~\ref{asympdalgmaxplimroot} gives $f_0 \in L_0$ with $a_\rho \pconv f_0$ and $P(f_0)=0$, and $f_1 \in L_1$ with $\mu(a_\rho) \pconv f_1$ and $P^\mu(f_1)=0$.
Now Lemma~\ref{adh11.4.8} gives an isomorphism $F_0\langle f_0 \rangle \cong F_1\langle f_1 \rangle$ extending $\mu$,
and we have a contradiction. Thus $F_0=L_0$ and
hence $F_1=L_1$. 
\end{proof}

\section*{Acknowledgements}

This research was supported by an NSERC Postgraduate Scholarship. 
Thanks are due to Lou van~den~Dries for many helpful comments on a draft of this paper.

\end{document}